\documentclass[12pt,a4paper]{article}

\usepackage{amsmath}
\usepackage{amsthm}
\usepackage{amssymb}
\usepackage{cain_arxiv}
\usepackage{hyperref}
\usepackage{booktabs}
\usepackage{pdflscape}
\usepackage{tikz}
\usetikzlibrary{positioning}

\theoremstyle{definition}
\newtheorem{definition}{Definition}[section]

\newtheorem{example}[definition]{Example}
\newtheorem{question}[definition]{Question}

\theoremstyle{plain}

\newtheorem{lemma}[definition]{Lemma}
\newtheorem{theorem}[definition]{Theorem}

\numberwithin{equation}{section}

\newcommand\fullref[2]{%
  \ifdefined\hyperref%
    {\hyperref[#2]{#1\ \penalty 200\relax\ref*{#2}}}%
  \else%
    {#1\ \penalty 200\relax\ref{#2}}%
  \fi%
}

\newcommand{\defterm}[1]{\textit{#1}}

\newcommand{\nset}{\mathbb{N}}
\newcommand{\zset}{\mathbb{Z}}

\DeclareMathOperator{\im}{im}
\DeclareMathOperator{\lcm}{lcm}

\newcommand{\rel}[1]{\mathcal{#1}}

\DeclareMathOperator{\gH}{\mathcal{H}}
\DeclareMathOperator{\gL}{\mathcal{L}}
\DeclareMathOperator{\gR}{\mathcal{R}}

\newcommand{\pres}[2]{\left\langle #1\:|\:#2 \right\rangle}
\newcommand{\sgpres}[2]{\mathrm{Sg}\langle #1\mid #2 \rangle}
\newcommand{\gpres}[2]{\mathrm{Gp}\langle #1\mid #2 \rangle}

\newcommand{\imreduces}{\rightarrow}
\newcommand{\reduces}{\rightarrow^*}

\begin{document}

%% \begin{center}
%% {\LARGE Hopfian and co-hopfian subsemigroups and extensions}

%% \bigskip
%% Alan J. Cain \\

%% Centro de Matem\'{a}tica, Faculdade de Ci\^{e}ncias, Universidade do Porto, \\
%% Rua do Campo Alegre 687, 4169--007 Porto, Portugal \\

%% email: ajcain@fc.up.pt \\

%% \medskip
%% Victor Maltcev \\

%% Mathematics Department, Technion -- Israel Institute of Technology, \\
%% Haifa 32000, Israel \\

%% email: victor.maltcev@gmail.com
%% \end{center}

\title{Hopfian and co-hopfian subsemigroups and extensions}
\author{Alan J. Cain \& Victor Maltcev}
\date{}

\thanks{The first author's research was funded by the European
  Regional Development Fund through the programme {\sc COMPETE} and by
  the Portuguese Government through the {\sc FCT} (Funda\c{c}\~{a}o
  para a Ci\^{e}ncia e a Tecnologia) under the project {\sc
    PEst-C}/{\sc MAT}/{\sc UI0}144/2011 and through an {\sc FCT}
  Ci\^{e}ncia 2008 fellowship. Some of the work described in this
  paper was carried out while the first author was visiting Sultan
  Qaboos University and the authors thank the University for its
  support.}

\maketitle

\address[AJC]{%
Centro de Matem\'{a}tica, Faculdade de Ci\^{e}ncias, \\
Universidade do Porto, \\
Rua do Campo Alegre 687, 4169--007 Porto, Portugal
}
\email{%
ajcain@fc.up.pt
}
\webpage{%
www.fc.up.pt/pessoas/ajcain/
}

\address[VM]{%
Mathematics Department,\\
Technion -- Israel Institute of Technology, \\
Haifa 32000, Israel
}
\email{%
victor.maltcev@gmail.com
}

\begin{abstract}
%% Keywords: hopfian, co-hopfian, subsemigroup, extension, finite index, Rees index, Green index.
%
%% MSC: 20M15 (Primary) 20M10 05C60 (Secondary).
%
This paper investigates the preservation of hopfic\-ity and
co-hopfic\-ity on passing to finite-index subsemigroups and
extensions. It was already known that hopfic\-ity is not preserved on
passing to finite Rees index subsemigroups, even in the finitely
generated case. We give a stronger example to show that it is not
preserved even in the finitely presented case. It was also known that
hopfic\-ity is not preserved in general on passing to finite Rees
index extensions, but that it is preserved in the finitely generated
case. We show that, in contrast, hopfic\-ity is not preserved on
passing to finite Green index extensions, even within the class of
finitely presented semigroups. Turning to co-hopfic\-ity, we prove
that within the class of finitely generated semigroups, co-hopfic\-ity
is preserved on passing to finite Rees index extensions, but is not
preserved on passing to finite Rees index subsemigroups, even in the
finitely presented case. Finally, by linking co-hopfic\-ity for graphs
to co-hopfic\-ity for semigroups, we show that without the hypothesis
of finite generation, co-hopfic\-ity is not preserved on passing to
finite Rees index extensions.
\end{abstract}

\section{Introduction}

An algebraic or relational structure is \defterm{hopfian} if it is not
isomorphic to any proper quotient of itself, or, equivalently, if any
surjective endomorphism of the structure is an automorphism. An
algebraic or relational structure is \defterm{co-hopfian} if it is not
isomorphic to any proper substructure of itself, or, equivalently, if
any injective endomorphism of the structure is an automorphism.

Hopfic\-ity was first introduced by Hopf, who asked if all finitely
generated groups were hopfian \cite{hopf_beitrage}. The celebrated
Baumslag--Solitar groups $\pres{x,y}{x^my = yx^n}$ provide the easiest
counterexample: $\pres{x,y}{x^2y = yx^3}$ is finitely generated, and
indeed finitely presented, and non-hopfian; see
\cite[Theorem~1]{baumslag_2generator}. Furthermore,
$\pres{x,y}{x^{12}y = yx^{18}}$ is hopfian but contains a non-hopfian
subgroup of finite index \cite[Theorem~2]{baumslag_2generator}. Hence
hopfic\-ity is not preserved under passing to finite-index
subgroups. On the other hand, a finite extension of a finitely
generated hopfian group is also hopfian
\cite[Corollary~2]{hirshon_hopficity}. There seems to have been no
study of whether co-hopfic\-ity for groups is preserved on passing to
finite-index subgroups or extensions, and it seems that the questions
of the preservation of co-hopficity in each direction are both
open. \fullref{Table}{tbl:summarygroups} summarizes the state of
knowledge about the preservation of hopficity and co-hopficity on
passing to finite-index subgroups and extensions.

\begin{table}[t]
\centering
\begin{tabular}{lc@{~}lc@{~}l}
  \toprule 
  & \multicolumn{4}{c}{Preserved on passing to finite index} \\
  \cmidrule{2-5}
  Property & \multicolumn{2}{c}{Subgroups} & \multicolumn{2}{c}{Extensions} \\
  \midrule
  Hopficity & N & [by f.p.~case] & ? & \fullref{Qu.}{qu:hopfgroupup}\\
  Hopficity \& f.g. & N & [by f.p.~case] & Y & \cite[Co.~2]{hirshon_hopficity} \\
  Hopficity \& f.p. & N & \cite[Th.~2]{baumslag_2generator} & Y & [by f.g.~result] \\
  Co-hopficity & ? & \fullref{Qu.}{qu:cohopfgroup} & ? & \fullref{Qu.}{qu:cohopfgroup}\\
  Co-hopficity \& f.g. & ? & \fullref{Qu.}{qu:cohopfgroup} & ? & \fullref{Qu.}{qu:cohopfgroup} \\
  Co-hopficity \& f.p. & ? & \fullref{Qu.}{qu:cohopfgroup} & ? & \fullref{Qu.}{qu:cohopfgroup} \\
  \bottomrule
\end{tabular}
\caption{Summary, for groups, of the preservation of hopficity and
  co-hopficity on passing to finite index subgroups and
  extensions. [\textit{Key:} f.g.~= finite generation; f.p.~= finite
  presentation; Y~= property is preserved; N~= property is not
  preserved in general; ?~= open question.]}
\label{tbl:summarygroups}
\end{table}

There are two useful notions of index for semigroups. For a semigroup
$S$ with a subsemigroup $T$, the \defterm{Rees index} of $T$ in $S$ is
$|S-T|+1$, and the \defterm{Green index} of $T$ in $S$ is the number
of $T$-relative $\gH$-classes in $S-T$. Rees index is more
established, and many finiteness properties, such as finite generation
and finite presentability, are known to be preserved on passing to or
from subsemigroups of finite Rees index; see the brief summary in
\cite[\S~11]{ruskuc_largesubsemigroups} or the comprehensive survey
\cite{cm_finreessurvey}. Green index is newer, but has the advantage
that finite Green index is a common generalization of finite Rees
index and finite group index, and some progress has been made in
proving the preservation of finiteness properties on passing to or
from subsemigroups of finite Green index; see
\cite{cgr_greenindex,gray_green1}.

The second author and Ru\v{s}kuc proved that a finite Rees index
extension of a finitely generated hopfian semigroup is itself hopfian
\cite[Theorem~3.1]{maltcev_hopfian}, and gave an example to show that
this no longer holds without the hypothesis of finite generation
\cite[\S~2]{maltcev_hopfian}. They also gave an example showing that
hopfic\-ity is not preserved on passing to finite Rees index
subsemigroups, even in the finitely generated case
\cite[\S~5]{maltcev_hopfian}. In this paper, we give an example
showing that it is not preserved even in the finitely presented case
(\fullref{Example}{eg:hopffreesfpdown}). We also give an example
showing that, again even in the finitely presented case, a finite
Green index extension of a hopfian semigroup need not be hopfian
(\fullref{Example}{eg:hopfgreenup}), showing that the result of the
second author and Ru\v{s}kuc does not generalize to finite Green
index.

We then turn to co-hopfic\-ity. We prove that a finite Rees index
extension of a finitely generated co-hopfian semigroup is itself
co-hopfian (\fullref{Theorem}{thm:cohopfreesfgup}), and construct an
example showing that this does not hold without the hypothesis of
finite generation (\fullref{Example}{eg:cohopfreesup}). We also give
an example of a non-co-hopfian finite Rees index subsemigroup of a
finitely presented co-hopfian semigroup
(\fullref{Example}{eg:cohopfreesdown}).

\fullref{Table}{tbl:summarysemigroups} summarizes the state of
knowledge for semigroups about the preservation of hopficity and
co-hopficity on passing to finite Rees and Green index subsemigroups
and extensions.

\begin{landscape}
\begin{table}
\centering
\begin{tabular}{lc@{~}lc@{~}lc@{~}lc@{~}l}
  \toprule 
                       & \multicolumn{8}{c}{Preserved on passing to}                                                                                                                                                                        \\
  \cmidrule{2-9}
                       & \multicolumn{4}{c}{Finite Rees index} & \multicolumn{4}{c}{Finite Green index}                                                                                                                                     \\
  \cmidrule(r){2-5} \cmidrule{6-9}
  Property             & \multicolumn{2}{c}{Subsemigroups}     & \multicolumn{2}{c}{Extensions}    & \multicolumn{2}{c}{Subsemigroups} & \multicolumn{2}{c}{Extensions}                                                                     \\
  \midrule
  Hopficity            & N                                     & [by f.g.~case]                    & N                                 & \cite[\S~2]{maltcev_hopfian}    & N & [by Rees index case] & N & [by Rees index case]              \\
  Hopficity \& f.g.    & N                                     & \cite[\S~5]{maltcev_hopfian}      & Y                                 & \cite[Th.~3.1]{maltcev_hopfian} & N & [by Rees index case] & N & [by f.p.~case]                    \\
  Hopficity \& f.p.    & N                                     & \fullref{Ex.}{eg:hopffreesfpdown} & Y                                 & [by f.g.~result]                & N & [by Rees index case] & N & \fullref{Ex.}{eg:hopfgreenup}
                                                                                                                                                                                                                                            \\
  Co-hopficity         & N                                     & [by f.p.~case]                    & N                                 & \fullref{Ex.}{eg:cohopfreesup}
                       & N                                     & [by Rees index case]              & N                                 & [by Rees index case]                                                                               \\
  Co-hopficity \& f.g. & N                                     & [by f.p.~case]                    & Y                                 & \fullref{Th.}{thm:cohopfreesfgup}
                       & N                                     & [by Rees index case]              & ?                                 & \fullref{Qu.}{qu:cohopfgreenfgup}                                                                  \\
  Co-hopficity \& f.p. & N                                     & \fullref{Ex.}{eg:cohopfreesdown}  & Y                                 & [by f.g.~result]                & N & [by Rees index case] & ? & \fullref{Qu.}{qu:cohopfgreenfgup} \\
  \bottomrule
\end{tabular}
\caption{Summary, for semigroups, of the preservation of hopficity and
  co-hopficity on passing to finite index subsemigroups and
  extensions. [\textit{Key:} f.g.~= finite generation; f.p.~= finite
  presentation; Y~= property is preserved; N~= property is not
  preserved in general; ?~= open question.]}
\label{tbl:summarysemigroups}
\end{table}
\end{landscape}

\section{Preliminaries}

\subsection{Presentations and rewriting systems}

The group presentation with (group) generators $A$ and defining
relations $\rel{R}$ (which may involve inverses of elements of $A$) is
denoted $\gpres{A}{\rel{R}}$. The semigroup presentation with
(semigroup) generators from $A$ and defining relations $\rel{R}$ is
denoted $\sgpres{A}{\rel{R}}$. For a semigroup $S$ presented by
$\sgpres{A}{\rel{R}}$ and words $u,v \in A^+$, write $u=v$ to indicate
that $u$ and $v$ are equal as words, and write $u =_S v$ to indicate
they represent the same element of $S$.

A \defterm{string rewriting system}, or simply a \defterm{rewriting
  system}, is a pair $(A,\rel{R})$, where $A$ is a finite alphabet and
$\rel{R}$ is a set of pairs $(\ell,r)$, often written $\ell \imreduces
r$, known as \defterm{rewriting rules}, drawn from $A^* \times
A^*$. The single reduction relation $\imreduces$ is defined
as follows: $u \imreduces_{\rel{R}} v$ (where $u,v \in A^*$) if there
exists a rewriting rule $(\ell,r) \in \rel{R}$ and words $x,y \in A^*$
such that $u = x\ell y$ and $v = xry$. The reduction relation
$\reduces$ is the reflexive and transitive closure of
$\imreduces$. A word $w \in A^*$ is \defterm{reducible} if
it contains a subword $\ell$ that forms the left-hand side of a
rewriting rule in $\rel{R}$; it is otherwise called
\defterm{irreducible}.

The string rewriting system $(A,\rel{R})$ is \defterm{Noetherian} if
there is no infinite sequence $u_1,u_2,\ldots \in A^*$ such that $u_i
\imreduces_{\rel{R}} u_{i+1}$ for all $i \in \nset$. The rewriting
system $(A,\rel{R})$ is \defterm{confluent} if, for any words $u,
u',u'' \in A^*$ with $u \reduces u'$ and $u
\reduces u''$, there exists a word $v \in A^*$ such that $u'
\reduces v$ and $u'' \reduces v$. A rewriting
system is complete if it is both confluent and Noetherian.

Let $(A,\rel{R})$ be a complete rewriting system. Then for any word $u
\in A^*$, there is a unique irreducible word $v \in A^*$ with $u
\reduces_{\rel{R}} v$ \cite[Theorem~1.1.12]{book_srs}. The irreducible
words are said to be in \defterm{normal form}. The semigroup presented
by $\sgpres{A}{\rel{R}}$ may be identified with the set of normal form
words under the operation of `concatenation plus reduction to normal
form'.

\subsection{Indices}

Let $S$ be a semigroup and let $T$ be a subsemigroup of $S$. The
\defterm{Rees index} of $T$ in $S$ is defined to be $|S-T|+1$. If $T$
is an ideal of $S$, then the Rees index of $T$ in $S$ is cardinality
of the Rees factor semigroup $S/T = (S - T) \cup \{0\}$.

To define the Green index of $T$ in $S$, we must first define the
$T$-relative Green's relations on $S$. As usual, $S^1$ denotes the
semigroup $S$ with an identity element adjoined. Extend this notation
to subsets of $S$: that is, $X^1 = X \cup \{ 1 \}$ for $X \subseteq
S$. Define the \defterm{$T$-relative Green's relations} $\gR^T$, $\gL^T$,
and $\gH^T$ on the semigroup $S$ by
\begin{align*}
x \gR^T y  &\Leftrightarrow  xT^1 = yT^1; & x \gL^T y &\Leftrightarrow T^1x = T^1y; &\gH^T &= \gR^T \cap \gL^T.
\end{align*}
Each of these relations is an equivalence relation on $S$. When $T =
S$, they coincide with the standard Green's relations on
$S$. Furthermore, these relations respect $T$, in the sense that each
$\gR^T$-, $\gL^T$-, and $\gH^T$-class lies either wholly in $T$ or
wholly in $S - T$. Following \cite{gray_green1}, define the
\defterm{Green index} of $T$ in $S$ to be one more than the number of
$\gH^T$-classes in $S - T$. If $S$ and $T$ are groups, then
$T$ has finite group index in $S$ if and only if it has finite Green
index in $S$ \cite[Proposition~8]{gray_green1}.

\section{Hopficity}

It is known that the hopfic\-ity is not preserved on passing to finite
Rees index subsemigroups, even for finitely generated semigroups
\cite[\S~5]{maltcev_hopfian}. The following example shows that within the
class of finitely \emph{presented} semigroups, and even within the
class of semigroups presented by finite complete rewriting systems,
hopfic\-ity is not preserved on passing to finite Rees
subsemigroups. This example has already appeared in the second
author's Ph.D. thesis \cite[Examples~5.6.1 \&~5.6.2]{maltcev_phd}.

\begin{example}
\label{eg:hopffreesfpdown}
Let
\begin{equation}
\label{eg:hopffreesfpdowntpres}
T = \sgpres{a,b}{abab^2ab = b}.
\end{equation}
Notice that
\[
abab^3 =_T abab^2(abab^2ab) = (abab^2ab)ab^2 ab =_T bab^2ab.
\]
It easy to check that the rewriting system $(\{a,b\},\{abab^2ab
\imreduces b, abab^3 \imreduces bab^2ab\})$ is confluent and
Noetherian. Clearly $T$ is also presented by
\[
\sgpres{a,b}{(abab^2ab,b),(abab^3,bab^2ab)}.
\]

Define an endomorphism
\[
\phi : T \to T; \qquad\qquad a \mapsto a,\qquad b\mapsto bab.
\]
This endomorphism is well defined since the words on the two sides of
the defining relation in the presentation
\eqref{eg:hopffreesfpdowntpres} for $T$ are mapped by $\phi$ to the
same element of $T$:
\begin{align*}
(abab^2ab)\phi =_T{}& a\;bab\;a\;(bab)^2\;a\;bab \\
={}& abab\;abab^2ab\;abab \\
\imreduces{}& abab^2ab\;ab \\
\imreduces{}& bab \\
={}& b\phi.
\end{align*}

Since $a\phi = a$ and
\begin{equation}
\label{eq:hopffreesfpdown1}
(ab^2)\phi =_T a(bab)^2 =_T abab^2ab \imreduces b,
\end{equation}
the endomorphism $\phi$ is surjective. Furthermore, applying
\eqref{eq:hopffreesfpdown1} shows that
\[
(ab^2a^2b^2)\phi =_T (ab^2\;a\;ab^2)\phi =_T bab =_T b\phi.
\]
But both $ab^2a^2b^2$ and $b$ are irreducible and so $ab^2a^2b^2
\neq_T b$. Hence $\phi$ is not bijective and so not an
automorphism. This proves that $T$ is not hopfian.

Let 
\[
S = \sgpres{a,b,f}{abab^2ab = b, fa=ba, af=ab, fb = bf = f^2 = b^2}.
\]
Notice that $S = T \cup \{f\}$ since all products of two or more
generators (regardless of whether they include generators $f$) must
lie in $T$. So $T$ is a finite Rees index subsemigroup of $S$. Notice
further that since $T$ is presented by a finite complete rewriting
system, so is $S$ \cite[Theorem~1]{wang_fcrssmallext}.

Let $\psi : S \to S$ be a surjective endomorphism. Since $a$ and $f$
are the only indecomposable elements of $S$, we have $\{a,f\}\psi =
\{a,f\}$. Let $\vartheta = \psi^2$; then $\vartheta$ is a surjective
endomorphism of $S$ with $a\vartheta = a$ and $f\vartheta = f$.

If $b\vartheta = f$, then $f =_S b\vartheta =_S (abab^2ab)\vartheta =_S
afaf^2af =_S abab^2ab =_S b$, which is a contradiction. Hence $b\vartheta = w \in T$. Then
\[
ab =_S af = (a\vartheta)(f\vartheta) =_S (af)\vartheta =_S (ab)\vartheta =
(a\vartheta)(b\vartheta) = aw.
\]
Now, $ab$ and $aw$ lie in the subsemigroup $T$ and so $ab =_T aw$.
But $T$ is left-cancellative by Adjan's theorem \cite{adjan_defining};
hence $b =_T w$ and so $b =_S w$. That is, $b\vartheta = w =_S
b$. Since $a\vartheta = a$ and $f\vartheta = f$, the endomorphism
$\vartheta$ must be the identity mapping on $S$ and so
bijective. Hence $\psi$ is bijective and so an automorphism. This
proves that $S$ is hopfian.

Therefore $S$ is a hopfian semigroup, finitely presented by a complete
rewriting system, with a non-hopfian subsemigroup $T$ of finite Rees
index, which is also finitely presented by a finite complete rewriting
system.
\end{example}

We now give an example to show that a finite Green index extension of
a finitely generated (and, indeed, finitely presented) hopfian
semigroup is not necessarily hopfian, in contrast to the situation for
finite Rees index \cite[Theorem~3.1]{maltcev_hopfian}.

\begin{example}
\label{eg:hopfgreenup}
Let $G$ and $H$ be the groups presented by
\begin{align*}
G &= \gpres{a,b,c}{a^{-1}ba = b^2, bc = cb}, \\
H &= \gpres{a',b',c'}{a'^{-1}b'a' = b'^2, b'c' = c'b', \\
&\qquad\qquad a'b'^{-1}a'^{-1}c'^{-1}a'b'a'^{-1}c'a'b'^{-1}a'^{-1}c'^{-1}a'b'a'^{-1}c' = 1}. 
\end{align*}
These groups were defined by Neumann
\cite[p.~543--4]{neumann_freeproducts}, except that he used redundant
generators $b'_1 = a'b'a'^{-1}$ and $d' = b_1'^{-1}c'^{-1}b_1'c'$ to
shorten the presentation of $H$; we have removed the redundant
generators to clarify the reasoning that follows. Let
\begin{align*}
\lambda &: G \to H; & a\lambda &= a', &b\lambda &= b', &c\lambda &= c'; \\
\mu &: H \to G; & a'\mu &= a, &b'\mu &= a^{-1}ba, &c'\mu &= c.
\end{align*}
The map $\lambda$ is obviously a well-defined surjective homomorphism;
Neumann \cite[p.~544]{neumann_freeproducts} showed that $\mu$ is also
a well-defined surjective homomorphism, and that neither $\lambda$ nor
$\mu$ is injective. That is, $G$ and $H$ are proper homomorphic images
of each other under the surjective homomorphisms $\lambda$ and
$\mu$. [Neumann defined $b'\mu = b^2$, but since $a^{-1}ba = b^2$ by
  the defining relations of $G$, our modified definition is
  equivalent.] Furthermore, $G$ and $H$ are non-isomorphic
\cite[Theorem on p.~544]{neumann_freeproducts}.

Let $\vartheta = \lambda\circ\mu$. Then $\vartheta : G \to G$ is a
surjective endomorphism of $G$ that is not an isomorphism. Notice that
\begin{align*}
a\vartheta &= a, &b\vartheta &= a^{-1}ba, &c\vartheta &= c.
\end{align*}

Let $F$ be the free group with basis $\{x,y,z\}$. Define
a homomorphism
\begin{align*}
\phi &: F \to G; & x\phi &= a, &y\phi &= b, &z\phi &= c.
\end{align*}
Partially order $\{F,G\}$ by $F > G$. Let $S$ be the Clifford
semigroup formed from the groups $F$ and $G$ with the order $\geq$ and
the homomorphism $\phi$. [See \cite[\S~4.2]{howie_fundamentals} for
the definition of Clifford semigroups.]

Clearly $F$ is a subsemigroup of $S$. Since the homomorphism $\phi$ is
surjective, any element of $G$ can be right-multiplied (in $S$) by an
element of $F$ to give any other element of $G$; thus all elements of
$G$ are related by $\gR^F$. Similarly all elements of $G$ are
$\gL^F$-related and so $\gH^F$-related. Therefore $G$ is the unique
$\gH^F$-class in $S-F$ and so $F$ has finite Green index in $S$.

Define an endomorphism
\begin{align*}
\psi &: S \to S; &x\psi &= x, &y\psi &= x^{-1}yx, &z\psi &= z,\\
&&a\psi &= a, &b\psi &= a^{-1}ba, &c\psi &= c.
\end{align*}
It is easy to see that $\psi$ is a homomorphism as a consequence of
$\psi|_G = \vartheta : G \to G$ being a homomorphism. Since $\psi|_G =
\vartheta$ is surjective, we have $G \subseteq \im\psi$. Since
$\{x,y,z\}\psi = \{x,x^{-1}yx,z\}$ generates $F$ (as a group), we see that $F \subseteq
\im\psi$. So $\psi$ is surjective. However, since $\psi|_G = \vartheta$
is not injective, $\psi$ is not injective. Hence $S$ is not hopfian.

Finally, note that the finitely generated free group $F$ is hopfian
\cite[Proposition~I.3.5]{lyndon_cgt}, and that $S$ is finitely
presented \cite[Theorem~5.1]{howie_constructions}. Therefore $F$ is a
finitely presented hopfian semigroup with a finitely presented
non-hopfian extension $S$ of finite Green index.
\end{example}

\section{Co-hopficity}

The following example exhibits a finitely generated co-hopfian
semigroup $S$ with a non-co-hopfian subsemigroup $T$ of finite Rees
index, showing that co-hopfic\-ity is not preserved on passing to finite
Rees index subsemigroups, even in the finitely generated (and, indeed,
finitely presented) case:

\begin{example}
\label{eg:cohopfreesdown}
Let $T$ be the free semigroup with basis $x$. Then any map $x \mapsto
x^k$ extends to an injective endomorphism from $T$ to itself; for $k
\geq 2$ this endomorphism is not bijective and so not an
automorphism. Thus $T$ is not co-hopfian.

Let
\[
S = \sgpres{x,y}{y^2 = xy = yx = x^2}
\]
Notice that $S = T \cup \{y\}$ since all products of two or more
generators must lie in $T$. So $T$ is a finite Rees index subsemigroup
of $S$. It is easy to check that the rewriting system $(\{x,y\},\{y^2
\imreduces x^2, xy \imreduces x^2, yx \imreduces x^2\})$ is confluent
and Noetherian. Identify $S$ with the set of irreducible words with
respect to this rewriting system. The Cayley graph of $S$ with respect
to $\{x,y\}$ is shown in
\fullref{Figure}{fig:cohopfreesdowncayleygraph}.

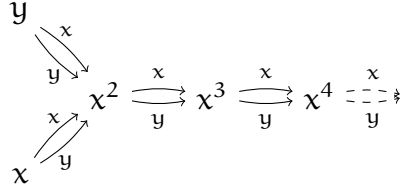
\begin{figure}[t]
\centerline{%
\begin{tikzpicture}[node distance=7mm]
\node (x1) {$x$};
\node (x2) [above right=of x1] {$x^2$};
\node (x3) [right=of x2] {$x^3$};
\node (x4) [right=of x3] {$x^4$};
\node (x5) [right=of x4] {};
\node (y1) [above left=of x2] {$y$};
\begin{scope}[every node/.style={font=\scriptsize}]
\draw[->] (x1) edge[bend left=10] node[anchor=south] {$x$} (x2);
\draw[->] (x1) edge[bend right=10] node[anchor=north] {$y$} (x2);
\draw[->] (y1) edge[bend left=10] node[anchor=south] {$x$} (x2);
\draw[->] (y1) edge[bend right=10] node[anchor=north] {$y$} (x2);
\draw[->] (x2) edge[bend left=10] node[anchor=south] {$x$} (x3);
\draw[->] (x2) edge[bend right=10] node[anchor=north] {$y$} (x3);
\draw[->] (x3) edge[bend left=10] node[anchor=south] {$x$} (x4);
\draw[->] (x3) edge[bend right=10] node[anchor=north] {$y$} (x4);
\draw[->] (x4) edge[bend left=10,dashed] node[anchor=south] {$x$} (x5);
\draw[->] (x4) edge[bend right=10,dashed] node[anchor=north] {$y$} (x5);
\end{scope}
\end{tikzpicture}}
\caption{The Cayley graph of the semigroup $S$ from
  \fullref{Example}{eg:cohopfreesdown} with respect to the generating
  set $\{x,y\}$.}
\label{fig:cohopfreesdowncayleygraph}
\end{figure}

Let $\phi : S \to S$ be an injective endomorphism. Suppose for
\textit{reductio ad absurdum} that $x\phi = x^k$ with $k \geq 2$. Then
$(y\phi)^2 = (y^2)\phi = (x^2)\phi = x^{2k}$, and so $y\phi = x^k =
x\phi$ since the unique square root of $x^{2k}$ in $S$ is $x^k$, which
contradicts the injectivity of $\phi$. Hence either $x\phi = x$ or
$x\phi = y$. In the former case, $x^\ell\phi = x^\ell$ for all $\ell
\in \nset$ and so $y\phi = y$ by the injectivity of $\phi$; hence
$\phi$ is surjective. In the latter case, $x^\ell\phi = y^\ell
\reduces x^\ell$ for all $\ell \geq 2$ and so $y\phi = x$ by the
injectivity of $\phi$; hence $\phi$ is surjective. In either case,
$\phi$ is a bijection and so an automorphism. Hence $S$ is co-hopfian.

Therefore $S$ is a co-hopfian semigroup presented by a finite complete
rewriting system, with a non-co-hopfian subsemigroup $T$ of finite
Rees index, which is also finitely presented by a finite complete
rewriting system (since it is free).
\end{example}

We have a positive result for passing to finite Rees index extensions
in the finitely generated case:

\begin{theorem}
\label{thm:cohopfreesfgup}
Let $S$ be a semigroup and $T$ a subsemigroup of $S$ of finite Rees
index. Suppose $T$ is finitely generated and co-hopfian. Then $S$ is co-hopfian.
\end{theorem}

Notice that, in \fullref{Theorem}{thm:cohopfreesfgup}, $S$ is also
finitely generated.

\begin{proof}
Let $X$ be a finite generating set for $T$ and let $\phi : S \to S$ be
an injective endomorphism. Let $t \in T$. Consider the images $t\phi$,
$t\phi^2$, $\ldots$. If $t\phi^i = t\phi^j$ for $i < j$, then the
injectivity of $\phi$ forces $t\phi^{j-i} = t$ and so
$t\phi^{\ell(j-i)} \in T$ for all $\ell \in \nset$. On the other hand,
if the elements $t\phi$, $t\phi^2$, $\ldots$ are all distinct, then
since $S-T$ is finite, $t\phi^\ell \in T$ for all sufficiently large
$\ell$. In either case, there exist some $k_t,m_t \in \nset$ such that
$t\phi^{\ell m_t} \in T$ for all $\ell \geq k_t$. Let $k = \max\{k_t :
t \in X\}$ and $m = \lcm\{m_t : t \in X\}$; both $k$ and $m$ exist
because $X$ is finite. Then $X\phi^{km} \subseteq T$, and so $T\phi^{km}
\subseteq T$ since $X$ generates $T$.

Since $\phi : S \to S$ is an injective endomorphism, so is $\phi^{km} : S
\to S$. Hence $\phi^{km}|_T$ is an injective endomorphism from $T$ to
$T$. Since $T$ is co-hopfian, $\phi^{km}|_T : T \to T$ is a
bijection. Therefore $\phi^{km}|_{S-T}$ must be an injective map from
$S-T$ to $S-T$, and hence a bijection since $S-T$ is finite. Thus
$\phi^{km} : S \to S$ is a bijection, and hence so is $\phi$.

Therefore any injective endomorphism from $S$ to itself is bijective
and so an automorphism. Thus $S$ is co-hopfian.
\end{proof}

We will shortly exhibit an example showing that
\fullref{Theorem}{thm:cohopfreesfgup} does not hold without the
hypothesis of finite generation. First, we need to define a
construction that builds a semigroup from a simple graph and establish some
of its properties.

\begin{definition}
\label{def:sgrpfromsimpgraph}
Let $\Gamma$ be a simple graph. Let $V$ be the set of vertices of $\Gamma$. Let $S_\Gamma = V \cup \{e,n,0\}$. Define a multiplication on $S_\Gamma$ by
\begin{align*}
v_1v_2 &=
\begin{cases} 
e & \text{if there is an edge between $v_1$ and $v_2$ in $\Gamma$,} \\
n & \text{if there is no edge between $v_1$ and $v_2$ in $\Gamma$,}
\end{cases} 
&& \text{for $v_1,v_2 \in V$,} \displaybreak[0]\\
ve &= ev = vn = nv = 0 && \text{for $v \in V$,} \displaybreak[0]\\
en &= ne = e^2 = n^2 = 0 \\
0x &= x0 = 0 && \text{for $x \in S_\Gamma$.}
\end{align*}
Notice that all products of two elements of $S_\Gamma$ lie in
$\{e,n,0\}$ and all products of three elements are equal to $0$. Thus
this multiplication is associative and $S_\Gamma$ is a semigroup.
\end{definition}

We emphasize that \fullref{Definition}{def:sgrpfromsimpgraph} only
applies to simple graphs.

\begin{lemma}
\label{lem:cofsubgraphifffinreesssg}
Let $\Gamma$ be a graph and let $\Delta$ be an induced subgraph of
$\Gamma$. Then the vertex set of $\Delta$ is cofinite in the vertex
set of $\Gamma$ if and only if $S_\Delta$ is a finite Rees index
subsemigroup of $S_\Gamma$.
\end{lemma}

\begin{proof}
Suppose $\Gamma$ has vertex set $V$ and $\Delta$ has vertex set
$W$. The result is immediate from the fact that $S_\Gamma - S_\Delta =
(V \cup \{e,n,0\}) - (W \cup \{e,n,0\}) = V- W$.
\end{proof}

The following lemma relates the co-hopficity of a graph $\Gamma$ and
the semigroup $S_\Gamma$. A homomorphism of graphs $\phi : \Gamma \to
\Gamma'$ is a mapping from the vertex set of $\Gamma$ to the vertex set of
$\Gamma'$ that preserves edges: that is, for all vertices $v_1$ and
$v_2$ of $\Gamma$, if $(v_1,v_2)$ is an edge of $\Gamma$, then
$(v_1\phi,v_2\phi)$ is an edge of $\Gamma'$. Note, however, that the
converse is not required to hold: it is possible that $(v_1,v_2)$ is
\emph{not} an edge of $\Gamma$, but $(v_1\phi,v_2\phi)$ \emph{is} an
edge of $\Gamma'$. As with other types of relational or algebraic
structure, a graph is co-hopfian if every injective endomorphism is an
automorphism.

\begin{lemma}
\label{lem:gammacohopfimpliessgammacohopf}
If the graph $\Gamma$ is co-hopfian, the semigroup $S_\Gamma$ is
co-hopfian.
\end{lemma}

\begin{proof}
Let $V$ be the vertex set of $\Gamma$ and let $X =
\{e,n,0\}$, so that $S_\Gamma = V \cup X$.

Suppose $\Gamma$ is co-hopfian; the aim is to show that $S_\Gamma$ is
co-hopfian. Let $\phi : S_\Gamma \to S_\Gamma$ be an injective
endomorphism. Since $X$ is the unique three-element null subsemigroup
of $S_\Gamma$, we have $X\phi = X$ and so $V\phi \subseteq V$ since
$\phi$ is injective. Furthermore, $0\phi = e^2\phi = (e\phi)^2 =
0$. Let $v \in V$; note that $v\phi \in V\phi \subseteq V$. Since
$\Gamma$ is simple, there are no loops at $v$ or $v\phi$, and so we
have $v^2 = n$ and $(v\phi)^2 = n$. Hence $n\phi = v^2\phi = (v\phi)^2
= n$. Therefore, since $X \phi = X$, it follows that $e\phi = e$. Let
$v_1,v_2 \in V$. Then
\begin{align*}
&\text{there is an edge between $v_1$ and $v_2$ in $\Gamma$} \\
\iff{}& v_1v_2 = e \\
\iff{}& (v_1v_2)\phi = e\phi \\
\iff{}& (v_1\phi)(v_2\phi) = e \\
\iff{}&\text{there is an edge between $v_1\phi$ and $v_2\phi$ in $\Gamma$.}
\end{align*}
Hence $\phi|_V : V \to V$ is an injective endomorphism of
$\Gamma$. Since $\Gamma$ is co-hopfian, $\phi|_V$ is a
bijection. Since $\phi|_X$ is a bijection, it follows that $\phi :
S_\Gamma \to S_\Gamma$ is a bijection. This proves that $S_\Gamma$ is
co-hopfian.
\end{proof}

\begin{lemma}
\label{lem:treesgammacohopfimpliesgammacohopf}
If $\Gamma$ is a tree and the semigroup $S_\Gamma$ is co-hopfian, the
graph $\Gamma$ is co-hopfian.
\end{lemma}

\begin{proof}
Let $V$ be the vertex set of $\Gamma$ and let $X =
\{e,n,0\}$, so that $S_\Gamma = V \cup X$.

Let $\Gamma$ be a tree and suppose that $S_\Gamma$ is co-hopfian; the
aim is to show $\Gamma$ is co-hopfian. Let $\phi : \Gamma \to \Gamma$
be an injective endomorphism. Extend $\phi$ to a map $\hat\phi :
S_\Gamma \to S_\Gamma$ by defining $n\hat\phi = n$, $e\hat\phi = e$,
and $0\hat\phi = 0$. Notice that $\hat\phi$ is injective since $\phi$
is injective. We now have to check the homomorphism condition for
$\hat\phi$ in various cases. Let $v_1,v_2 \in V$. Then either $v_1v_2
= e$ or $v_1v_2 = n$; we consider these cases separately:
\begin{itemize}
\item If $v_1v_2 = e$, then there is an edge between $v_1$ and $v_2$
  in $\Gamma$, and so, since $\phi$ is an endomorphism of $\Gamma$,
  there is an edge between $v_1\phi$ and $v_2\phi$ in $\Gamma$, and
  thus $(v_1\phi)(v_2\phi) = e$. Therefore $v_1v_2 = e$ implies
  $(v_1v_2)\hat\phi = e\hat\phi = e = (v_1\phi)(v_2\phi) =
  (v_1\hat\phi)(v_2\hat\phi)$.
\item If $v_1v_2 = n$, then there is no edge between $v_1$ and
  $v_2$. Since $\Gamma$ is a tree and thus connected, there is a path
  $\pi = (v_1 = x_1,x_2,\ldots,x_n=v_2)$ from $v_1$ to $v_2$. Since
  there is no edge between $v_1$ and $v_2$, we have $n \geq 3$. Since
  $\phi$ is an injective endomorphism of $\Gamma$, there is a path
  $\pi\phi = (v_1\phi = x_1\phi,x_2\phi,\ldots,x_n\phi=v_2\phi)$ from
  $v_1\phi$ to $v_2\phi$. In particular, injectivity means that
  $v_1\phi$ and $v_2\phi$ are not among the intermediate vertices
  $x_2\phi,\ldots,x_{n-1}\phi$. Now, if there were an edge between
  $v_1\phi$ and $v_2\phi$, then this edge and the path $\pi\phi$ would
  form a non-trivial cycle, contradicting the fact that $\Gamma$ is a
  tree. Hence there is no edge between $v_1\phi$ and
  $v_2\phi$. Therefore $v_1v_2 = n$ implies $(v_1v_2)\hat\phi =
  n\hat\phi = n = (v_1\phi)(v_2\phi) = (v_1\hat\phi)(v_2\hat\phi)$.
\end{itemize}
Since any product where at least one of the element is not from $V$ is
equal to $0$, it is easy to see that the endomorphism condition holds
in these cases. Hence $\hat\phi : S_\Gamma \to S_\Gamma$ is an
injective endomorphism. Since $S_\Gamma$ is co-hopfian, $\hat\phi$ is
a bijection, and so $\phi = \hat\phi|_V$ is a bijection. This proves
that $\Gamma$ is co-hopfian.
\end{proof}

We can now present the example showing that
\fullref{Theorem}{thm:cohopfreesfgup} no longer holds without the
hypothesis of finite generation:

\begin{example}
\label{eg:cohopfreesup}
Define a graph $\Gamma$ as follows. The vertex set is
\[
V = \{x_i, y_i: i \in \zset\} \cup \{z_j : j \in \nset\},
\]
and there are edges between $x_i$ and $y_i$ for all $i \in \zset$,
between $y_j$ and $z_j$ for all $j \in \nset$, and between $x_i$ and
$x_{i+1}$ for all $i \in \zset$. The graph $\Gamma$ is as shown in
\fullref{Figure}{fig:graphgamma}. Let $\Delta$ be the subgraph induced
by $W = V - \{y_0\}$; the graph $\Delta$ is as shown in
\fullref{Figure}{fig:graphgammadash}. Note that $\Gamma$ and $\Delta$
are trees and in particular simple.

\begin{figure}[t]
\centerline{%
\begin{tikzpicture}[node distance=5mm]
\node (x0) {$x_0$};
\node (x1) [right=of x0] {$x_1$};
\node (x2) [right=of x1] {$x_2$};
\node (x3) [right=of x2] {$x_3$};
\node (x4) [right=of x3] {};
\node (xm1) [left=of x0] {$x_{-1}$};
\node (xm2) [left=of xm1] {$x_{-2}$};
\node (xm3) [left=of xm2] {$x_{-3}$};
\node (xm4) [left=of xm3] {};
\node (y0) [above=of x0] {$y_0$};
\node (y1) [above=of x1] {$y_1$};
\node (y2) [above=of x2] {$y_2$};
\node (y3) [above=of x3] {$y_3$};
\node (ym1) [above=of xm1] {$y_{-1}$};
\node (ym2) [above=of xm2] {$y_{-2}$};
\node (ym3) [above=of xm3] {$y_{-3}$};
\node (z1) [above=of y1] {$z_1$};
\node (z2) [above=of y2] {$z_2$};
\node (z3) [above=of y3] {$z_3$};
\draw[-] (x0) edge (x1);
\draw[-] (x1) edge (x2);
\draw[-] (x2) edge (x3);
\draw[-] (x3) edge[dashed] (x4);
\draw[-] (x0) edge (xm1);
\draw[-] (xm1) edge (xm2);
\draw[-] (xm2) edge (xm3);
\draw[-] (xm3) edge[dashed] (xm4);
\draw[-] (x0) edge (y0);
\draw[-] (x1) edge (y1);
\draw[-] (x2) edge (y2);
\draw[-] (x3) edge (y3);
\draw[-] (xm1) edge (ym1);
\draw[-] (xm2) edge (ym2);
\draw[-] (xm3) edge (ym3);
\draw[-] (y1) edge (z1);
\draw[-] (y2) edge (z2);
\draw[-] (y3) edge (z3);
\end{tikzpicture}}
\caption{The graph $\Gamma$ from \fullref{Example}{eg:cohopfreesup}.}
\label{fig:graphgamma}
\end{figure}

\begin{figure}[t]
\centerline{%
\begin{tikzpicture}[node distance=5mm]
\node (x0) {$x_0$};
\node (x1) [right=of x0] {$x_1$};
\node (x2) [right=of x1] {$x_2$};
\node (x3) [right=of x2] {$x_3$};
\node (x4) [right=of x3] {};
\node (xm1) [left=of x0] {$x_{-1}$};
\node (xm2) [left=of xm1] {$x_{-2}$};
\node (xm3) [left=of xm2] {$x_{-3}$};
\node (xm4) [left=of xm3] {};
\node (y1) [above=of x1] {$y_1$};
\node (y2) [above=of x2] {$y_2$};
\node (y3) [above=of x3] {$y_3$};
\node (ym1) [above=of xm1] {$y_{-1}$};
\node (ym2) [above=of xm2] {$y_{-2}$};
\node (ym3) [above=of xm3] {$y_{-3}$};
\node (z1) [above=of y1] {$z_1$};
\node (z2) [above=of y2] {$z_2$};
\node (z3) [above=of y3] {$z_3$};
\draw[-] (x0) edge (x1);
\draw[-] (x1) edge (x2);
\draw[-] (x2) edge (x3);
\draw[-] (x3) edge[dashed] (x4);
\draw[-] (x0) edge (xm1);
\draw[-] (xm1) edge (xm2);
\draw[-] (xm2) edge (xm3);
\draw[-] (xm3) edge[dashed] (xm4);
\draw[-] (x1) edge (y1);
\draw[-] (x2) edge (y2);
\draw[-] (x3) edge (y3);
\draw[-] (xm1) edge (ym1);
\draw[-] (xm2) edge (ym2);
\draw[-] (xm3) edge (ym3);
\draw[-] (y1) edge (z1);
\draw[-] (y2) edge (z2);
\draw[-] (y3) edge (z3);
\end{tikzpicture}}
\caption{The cofinite subgraph $\Delta$ of the graph $\Gamma$ from
  \fullref{Example}{eg:cohopfreesup}.}
\label{fig:graphgammadash}
\end{figure}

Define a map
\begin{align*}
\phi {}&: V \to V; & x_i & \mapsto x_{i+1} &&\text{for all $i \in \zset$,}\\
&& y_i & \mapsto y_{i+1} &&\text{for all $i \in \zset$,}\\
&& z_i & \mapsto z_{i+1} &&\text{for all $i \in \nset$.}
\end{align*}
It is easy to see that $\phi$ is an injective endomorphism of
$\Gamma$. However, $\phi$ is not a bijection since $z_1 \notin
\im\phi$. Thus the graph $\Gamma$ is not co-hopfian.

Suppose $\psi : W \to W$ is an injective endomorphism of
$\Delta$. Clearly $\psi$ must preserve adjacency in $\Delta$. So the
bi-infinite path through the vertices $x_i$ must be mapped into
itself. The preservation of adjacency requires that this path is
mapped \emph{onto} itself. All vertices on this path have degree $3$
except $x_0$. Hence $x_0\psi = x_0$. The preservation of adjacency
requires that either $x_{i}\psi = x_{-i}$ or $x_{i}\psi = x_i$ for all
$i \in \zset$. The former case is impossible since it would force
$y_i\psi = y_{-i}$ for all $i \in \zset$, but $y_1$ has degree
$2$ and $y_{-1}$ has degree $1$. Hence the latter case holds, which
forces $y_i\psi = y_i$ for all $i \in \zset$, and then $z_j\psi = z_j$
for all $j \in \nset$. Hence $\psi$ is the identity map and so
bijective. Thus the subgraph $\Delta$ is co-hopfian.

By \fullref{Lemma}{lem:cofsubgraphifffinreesssg}, $S_\Delta$ is a
finite Rees index subsemigroup of $S_\Gamma$. By
\fullref{Lemma}{lem:gammacohopfimpliessgammacohopf}, $S_\Delta$ is
co-hopfian. By
\fullref{Lemma}{lem:treesgammacohopfimpliesgammacohopf}, $S_\Gamma$ is
not co-hopfian.
\end{example}

\section{Open questions}

For semigroups, the main open problem in this area seem to be whether
\fullref{Theorem}{thm:cohopfreesfgup} generalizes to finite Green
index extensions:

\begin{question}
\label{qu:cohopfgreenfgup}
Let $S$ be a semigroup and $T$ a subsemigroup of finite Green
index. Suppose $T$ is finitely generated (or even finitely presented),
so that $S$ is finitely generated
\cite[Theorem~4.3]{cgr_greenindex}. If $T$ is co-hopfian, must $S$ be
co-hopfian?
\end{question}

Notice that because finite Green index generalizes finite group index,
this question subsumes the corresponding question for group
extensions.

Since relative finiteness and finite presentability are not preserved
on passing to finite Green index extensions unless the relative
Sch\"{u}tzenberger groups of the relative $\gH$-classes in the
complement have the relevant property (see
\cite[Theorem~20]{gray_green1} and
\cite[Example~6.5]{cgr_greenindex}), it is natural to ask the
following question:

\begin{question}
Let $S$ be a semigroup and $T$ a subsemigroup of finite Green
index. Suppose $T$ is finitely generated, so that $S$ is finitely
generated \cite[Theorem~4.3]{cgr_greenindex} and the the $T$-relative
Sch\"{u}tzenberger groups of the $\gH^T$-classes in $S-T$ are finitely
generated \cite[Theorem~5.1]{cgr_greenindex}. If $T$ is hopfian, and
the $T$-relative Sch\"{u}tzenberger groups of the $\gH^T$-classes in
$S-T$ are hopfian, must $S$ be hopfian?

If the answer to \fullref{Question}{qu:cohopfgreenfgup} is `no', then
the question from the previous paragraph should be asked for
co-hopfic\-ity: if $T$ is co-hopfian, and the $T$-relative
Sch\"{u}tzenberger groups of the $\gH^T$-classes in $S-T$ are
co-hopfian, must $S$ be co-hopfian?
\end{question}

For groups, the following question still seems to be open:

\begin{question}
\label{qu:hopfgroupup}
Is hopficity for groups preserved under passing to finite index
extensions? (That is, does Hirshon's result
\cite[Corollary~2]{hirshon_hopficity} hold without the hypothesis of
finite generation?)
\end{question}

Finally, none of the relevant questions on co-hopficity for groups
have been studied:

\begin{question}
\label{qu:cohopfgroup}
Is co-hopficity for groups preserved under passing to finite index
subgroups and extensions? What about within the classes of finitely
generated or finitely presented groups?
\end{question}

\section*{Acknowledgements}

The first author's research was funded by the European Regional
Development Fund through the programme {\sc COMPETE} and by the
Portuguese Government through the {\sc FCT} (Funda\c{c}\~{a}o para a
Ci\^{e}ncia e a Tecnologia) under the project {\sc PEst-C}/{\sc
  MAT}/{\sc UI0}144/2011 and through an {\sc FCT} Ci\^{e}ncia 2008
fellowship. Some of the work described in this paper was carried out
while the first author was visiting Sultan Qaboos University and the
authors thank the University for its support.

The authors thank the anonymous referee for pointing out an error in
one proof and for valuable comments and suggestions.

%% \bibliography{c_publications,groups,semigroups,presentations,\jobname}
%% \bibliographystyle{alphaabbrv}

\end{document}